\documentclass[11pt,reqno]{amsart}
\usepackage{amssymb}
\usepackage{amsmath}
\usepackage{graphics}
\usepackage{epsfig}

\newcommand{\RR}{\mathbb R}

\newfont{\fung}{cmff10 at 15pt}

\newtheorem{theorem}{Theorem}[section]
\newtheorem{lemma}[theorem]{Lemma}

\def\thebiblio#1{\subsection*{References}\list
{[\arabic{enumi}]}{\settowidth\labelwidth{#1.}
\leftmargin\labelwidth
\advance\leftmargin\labelsep\usecounter{enumi}}
\def\newblock{\hskip .11em plus .33em minus -.07em}\sloppy
\sfcode`\.=1000\relax}

\begin{document}

\title[]{Well-posedness and stabilization of a model system
for long waves posed on a quarter plane}






\author{A. F. Pazoto}
\address{Instituto de Matem\'atica, Universidade Federal do Rio de
Janeiro, P.O. Box 68530, CEP 21945-970, Rio de Janeiro, RJ,  Brasil}
\email{ademir@im.ufrj.br}

\author{G. R. Souza}
\address{Centro Federal de Educa\c{c}\~ao Tecnol\'ogica Celso Suckow da Fonseca, Avenida Governador Roberto Silveira, 1900, CEP 28.635-000, Nova Friburgo, RJ, Brasil}
\email{gilmar@im.ufrj.br}


\begin{abstract}
In this paper we are concerned with a initial boundary-value problem for a coupled system of two KdV equations, posed on the positive half line, under the effect of a localized damping term. The model arises when modeling  the  propagation  of  long  waves  generated by a wave maker in a channel. It is shown that the solutions of the system are exponential stable and globally well-posed in the weighted space $L^2(e^{2bx}dx)$ for $b>0$. The stabilization problem is studied using a Lyapunov approach while the well-posedness result is obtained combining fixed point arguments and energy type estimates.
\end{abstract}

\maketitle

\section{Introduction}
We consider the initial-boundary value problem for a coupled system of two Korteweg-de Vries equation posed on the positive half line under the presence of a localized damping represented by a function $a$, that is,
\begin{equation}\label{e1}
\left\{
\begin{array}{l}%
u_{t}+u_{xxx}+a_{3}v_{xxx}+uu_{x}+a_{1}vv_{x}+a_{2}(uv)_{x}+a(x)u=0, \quad x, t\in \mathbb{R}^+\\
b_{1}v_{t}+v_{xxx}+b_{2}a_{3}u_{xxx}+vv_{x}+b_{2}a_{2}uu_{x}+b_{2}
a_{1}(uv)_{x}+rv_{x}+a(x)v=0\\
u(0,t)=v(0,t)=0, \quad t>0 \\
u(x,0)=u_{0}(x), v(x,0)=v_{0}(x), \quad x\in \mathbb{R}^+,\\
\end{array}
\right.
\end{equation}
where $r, a_1, a_2, a_3, b_1, b_2$ are real constants with $b_1, b_2 > 0$.

When $a\equiv 0$ system \eqref{e1} was derived by Gear
and Grimshaw in \cite{gg} as a model to describe strong interactions
of two long internal gravity waves in a stratified fluid, where the
two waves are assumed to correspond to different modes of the
linearized equations of motion. It has the structure of a pair of
KdV equations with both linear and nonlinear coupling terms and has
been object of intensive research in recent years (see, for instance, \cite{ablowitz,alarcon,vep,bona,cerpa-pazoto,davila,li,mp,mo,mop,dugan,pazoto-souza,saut,vera}). This apparently complicated system also appears as a special case of a broad class of nonlinear evolution equations that can also be interpreted as a coupled nonlinear version of Korteweg-de Vries generalized equations of the form
\begin{equation}
\left\{
\begin{array}{l}
u_t + u_{xxx} + f(u,v)_x=0\\
v_t + v_{xxx} + g(u,v)_x=0,\nonumber
\end{array}
\right.
\end{equation}
with $f$ and $g$ satisfying $f(u,v)=H_u(u,v)$ and $g(u,v)=H_v(u,v)$ for a smooth function $H$.
The initial boundary value problems for KdV type models arise
naturally in modeling small-amplitude long waves in a channel with a wavemaker mounted at one end, or in modeling coastal zone motions
generated by long-crested waves propagating shoreward from deep water (see, for instance, \cite{bona-winther}).
Such mathematical formulations have received considerable attention in the past, and a satisfactory theory
on the global well-posedness of solutions is available in the literature.

Along this work we assume that the real-valued function $a =a(x)$ satisfies the condition
\begin{equation}\label{a}
\begin{array}{l}
a \in L^\infty (\mathbb{R}^+) \,\hbox { is a positive function and }\, a(x) \geq a_0 > 0 \, \hbox { a. e. in } \, \omega,
\end{array}
\end{equation}
where $a_0$ is sufficiently large and $\omega$ contains two sets of the form $(0,\delta)$ and $(\beta,\infty)$, for any $0< \delta << 1$ and $\beta > 0$. Observe that in this case the damping term is acting effectively in $\omega$.

We also assume that
$$r, a_1, a_2, a_3, b_1 \mbox{ and } b_2 \mbox{ are real constants with } a_1=a_2,\, 0 < r,\,a^2_3b_2 < < 1\mbox{ and } b_1, b_2 > 0.$$
Assumption $0 < a^2_3b_2 < 1$ combined with some conservations law satisfied by the solutions allow to obtain a priori estimates leading to the global well-posedness results. The constant $r$ is a non-dimensional constant parameter that can be assumed to be very small (see, for instance, \cite{bona,saut}).
Condition $a_1=a_2$ is technical and will be used to simplify some computations.

In order to make more precise the idea we have in mind, let us first consider the total energy associated to \eqref{e1}, in this case
\begin{equation}\label{energy}
E(t) = \frac{1}{2}\int_{\mathbb{R^+}} (b_2u^2 + b_1v^2) dx.
\end{equation}
Under the above conditions, we have
\begin{equation}\label{dissipation}
\begin{array}{l}
\vspace{1mm}\displaystyle\frac{d}{dt}E(t)= - [\,\frac{b_2}{2}u_x^2(0,t) + \frac{1}{2}v_x^2(0,t) +
a_3b_2u_x(0,t)v_x(0,t)]-\displaystyle\int_{\mathbb{R}^+} a(x)(b_2u^2 + v^2) dx\\
=
-\displaystyle\frac{1}{2}\left(\sqrt{b_{2}}u_{x}(0,t)+\sqrt{a^{2}_{3}b_{2}}v_{x}(0,t)\right)^{2}
-\displaystyle\frac{1}{2}\left(1-a^{2}_{3}b_{2}\right)v^{2}_{x}(0,t)-\displaystyle\int_{\mathbb{R}^+} a(x)(b_2u^2 + v^2)
    dx \leq 0.
\end{array}
\end{equation}
This indicates that terms $a(x)u$ and $a(x)v$ in the equations plays the role of a feedback damping mechanism.
So, in the light of the computations performed above, we can see that \eqref{e1} generates a flow that can be
continued indefinitely in the temporal variable, defining a solution valid for all $t\geq 0$. Moreover, taking \eqref{dissipation} into account,
the internal stabilization problem was addressed by the authors in \cite{pazoto-souza}
and a locally exponential decay result was derived. In this context, it is important to note that, even when
$a\equiv 0$ or $\omega = \emptyset$, the energy is decreasing and is strictly so as long as $u_x(0,t), v_x(0,t)\neq 0$.
This suggests that any solution of the model may converges to zero exponentially in the absence of the internal damping.
However, this may not always be the case, since the dissipation due to the boundary terms does not seem to be strong enough
to guarantee the exponential decay of solutions, including the case of the scalar KdV equation. In addition, negative result probably holds when the support of $a$ is a finite interval.


Here, two mathematical issues connected to \eqref{e1} will be addressed; well-posedness and large-time behavior of solutions in a weighted space introduced by Kato in \cite{kato}. More precisely, for $b>0$ we prove that the solutions are globally well-posed and decay exponentially to zero in $[L^2_b(\mathbb{R}^+)]^2$, where
 \begin{equation}\label{kato}
 L^2_b(\mathbb{R}^+)=\{u : \mathbb{R}^+ \rightarrow \mathbb{R}  ; \|u e^{bx}\|_{L^2(\mathbb{R}^+)} < \infty \}.
 \end{equation}
To our knowledge, both problems have not been addressed in the literature yet and the existing results do not give them an immediate answer. As it is known in the theory of dispersive wave equations, global well-posedness in any particular space seems to depend on the available local theory, on the energy- type inequalities satisfied by the solutions and also on the linear theory. Therefore, we first show that the corresponding linear problem generates a semigroup of continuous operator in $[L^2_b(\mathbb{R}^+)]^2$. In particular, we establish the so-called Kato smoothing effect, i. e.,  the solutions whose initial datum lies in $[L^2_b(\mathbb{R}^+)]^2$ not only lies in $\mathcal{C}([0,T], [L^2_b(\mathbb{R}^+)]^2)$, but also in $L^2(0,T;[H^1_b(\mathbb{R}^+)]^2)$. This property made it possible to combine Duhamel formula and a contraction-mapping argument to prove directly the local well-posedness result. In order to get the global result we derive some a priori global estimate, whence we also deduce the Kato smoothing effect for the nonlinear problem. Those a priori estimates were not available. In fact, the only available a priori estimates for the initial-boundary value problem (the energy estimate and the kato effect) are not sufficient to yield a global well-posedness result.

With the well-posedness results in hand, we can investigate the long time
behavior of solutions. More precisely, we aim to prove that the localized damping mechanism $a(x)(u,v)$ is strong enough in order to guarantee the exponential decay of solutions in $[L^2_b(\mathbb{R}^+)]^2$, as
long as the initial data in \eqref{e1} live inside a ball of radius R (no matter how large is
$R > 0$). This is done introducing a convenient Lyapunov function $\mathcal{L}=\mathcal{L}(U(t))$ function, which is a suitable perturbation of the $[L^2_b(\mathbb{R}^+)]^2$-norm of the solution, that satisfies
$$\mathcal{L}(U(t)) - \mathcal{L}(U_0) \leq c\,\mathcal{L}(U_0),\mbox{ for some } c>0.$$
Indeed, this inequality gives at once the decay
$$\mathcal{L}(U(t)) \leq C \mathcal{L}(U_0) e^{-\nu t},$$
where $C$ and $\nu$ are positive constant with $\nu=\nu(E(0))$. The proof combines energy estimates and the result on the exponential decay of $E(t)$ obtained in \cite{pazoto-souza}. It is also worth mentioning the work \cite{pazoto-rosier} that develops the same analysis for the corresponding scalar KdV equation and from which we borrow the ideas involved in our proofs. However, since we are concerned with a coupled system of nonlinear equations we need more delicate estimates. Moreover, we improve the analysis developed in \cite{pazoto-rosier} showing that the inequality above holds without any assumption on $b$.

Our result on the time decay rate is of local nature in the sense that the exponential decay rate is uniform
only in bounded sets of initial data. However, the results obtained in this
paper do not provide any estimate on how the decay rate depends on the radius $R$ of
the ball. This has been done for nonlinear models, as far as we know, in very few cases
and always using some structural conditions on the nonlinearity.

We finish this section mentioning some results obtained for the scalar KdV equation in connection with the analysis developed here \cite{cavalcanti,laurent,linares-pazoto,linares-pazoto1,massarolo-menzala-pazoto1,pazoto,rosier1,rosier2,rosier3,rosier-zhang}. In what concern system \eqref{e1}, except in \cite{pazoto-souza}, the problem on internal stabilization was addressed only on a bounded domain \cite{vep,massarolo-menzala-pazoto,dugan}. In fact, most of results present in these works are proved following the ideas introduced for the analysis of the corresponding scalar case.

The analysis described above was organized in two sections. In Section 2 we present the global well-posedness result in two parts, the first one is devoted to the linear theory. Section 3 is devoted the stabilization problem. In both sections we split the results into several steps in order to make the reading easier.

\section{Well-posedness}

\setcounter{equation}{0}

The key of the modern analysis of nonlinear, dispersive evolution equations is the linear estimates. Therefore, first we address the well-posedness result for linear problem associated to the initial-boundary value problem \eqref{e1}.

\subsection{The linear system} Fix $b>0$. Considerations is first directed to the corresponding
linear system
\begin{equation}\label{sis lin}
\left\{
\begin{array}{l}
u_{t}+u_{xxx}+a_{3}v_{xxx}+a(x)u=0\\
b_{1}v_{t}+v_{xxx}+b_{2}a_{3}u_{xxx}+rv_{x}+a(x)v=0,\;x,t\in \RR^+,\\
u(0,t)=v(0,t)=0,\; t>0,\\
u(x,0)=u_{0}(x), v(x,0)=v_{0}(x),\;x>0.
\end{array}
\right.
\end{equation}\\
In order to apply the semigroup theory, we introduce the operator
$$A: D(A)\subset \left[L^2_b(\RR^+)\right]^2\rightarrow \left[L^2_b(\RR^+)\right]^2$$
defined by
$$
A\left(
\begin{array}[c]{l}
u\\
v
\end{array}
\right)=\left(
\begin{array}{l}
-u_{xxx}-a_3v_{xxx}\\
-\dfrac{1}{b_1}v_{xxx}-\dfrac{a_3b_2}{b_1}u_{xxx}-\dfrac{r}{b_1}v_x
\end{array}
\right)
$$
with the domain
$$D(A):=\left\{U= \left(
\begin{array}[c]{l}
u\\
v
\end{array}
\right) \in\left[L^2_b(\RR^+)\right]^2\; : \partial^i_xU\in\left[L^2_b(\RR^+)\right]^2\mbox{ for } 1\leq i \leq 3\mbox{ and } U(0)=0\right\}.$$

With the notation introduced above, the following holds:

\begin{lemma}\label{exist linear}
The operator $A$ defined above is the infinitesimal generator of a continuous semigroup of operators $(S(t))_{t\geq0}$ in $\left[L^2_b(\RR^+)\right]^2$.
\end{lemma}
\begin{proof}
We first make the change of variables $p=e^{bx}u$ and $q=e^{bx}v$. Then, if $u$ and $v$ solves \eqref{sis lin}, $p$ and $q$ solve
\begin{equation}\label{sis lin 2}
\left\{
\begin{array}{l}
p_{t}+(\partial_x-b)^3p+a_3(\partial_x-b)^3q+a(x)p=0\\
b_{1}q_{t}-rbq+rq_x+(\partial_x-b)^3q+a_3b_2(\partial_x-b)^3p+a(x)q=0 \ \ \ x,t\in \RR^+,\\
p(0,t)=q(0,t)=0,\ \ t>0,\\
p(x,0)=p_0(x)=e^{bx}u_{0}(x);\ \ q(x,0)=q_0(x)=e^{bx}v_{0}(x),\; x>0.
\end{array}
\right.
\end{equation}\\
Now we introduce the Hilbert space
$$X=\left[L^2(\RR^+)\right]^2$$
endowed with the inner product
$$ \left(\left(\begin{array}[c]{l}
p\\
q
\end{array}\right),\left(\begin{array}[c]{l}
\phi\\
\psi
\end{array}\right)\right)_X=\frac{b_2}{b_1}\int_{\mathbb{R}^+} p\phi dx+\int_{\mathbb{R}^+} q\psi dx.$$
Next, we consider the operator
$$B: D(B)\subset X\rightarrow X$$
with the domain
$$
D(B)=\left\{W=\left(
\begin{array}[c]{l}
p\\
q
\end{array}
\right)\in[H^3(\RR^+)]^2:W(0)=0 \right\}$$
defined by
\begin{equation}
BW=B\left(
\begin{array}{l}
p\\
q
\end{array}
\right)\\
=\left(
\begin{array}{l}
-(\partial_x-b)^3p-a_3(\partial_x-b)^3q-a(x)p\\
-(\partial_x-b)\dfrac{r}{b_1}q-\dfrac{a_3b_2}{b_1}(\partial_x-b)^3p-(\partial_x-b)^3\dfrac{1}{b_1}q-\dfrac{a(x)}{b_1}q
\end{array}
\right).
\end{equation}
Under the above conditions, the adjoint of the operator $B$ is the operator
$B^*$ given by
\begin{equation}
B^*\left(
\begin{array}{l}
p\\
q
\end{array}
\right)\\
=\left(
\begin{array}{l}
(\partial_x+b)^3p+a_3(\partial_x+b)^3q-a(x)p\\
(\partial_x +b)\dfrac{r}{b_1}q+\dfrac{a_3b_2}{b_1}(\partial_x +b)^3p+(\partial_x +b)^3\dfrac{1}{b_1}q-\dfrac{a(x)}{b_1}q
\end{array}
\right),
\end{equation}
where
$$B^*:D(B^*)\subset X\rightarrow X$$
and
$$D(B^*)=\left\{W=\left(
\begin{array}[c]{l}
p\\
q
\end{array}
\right)\in [H^3(\RR^+)]^2:W(0)=\partial_xW(0)=0\right\}.$$
Since $B$ is a closed operator $(B^{**}=B)$ and the domains $D(B)$ and $D(B^*)$ are dense in $X$, the result is proved if the operators $B-\lambda I$ and its adjoint $B^*-\lambda I$ are both dissipative for some  $\lambda\in \RR^+$. Indeed, first pick $W=\left(
\begin{array}{l}
p\\
q
\end{array}
\right) \in D(B)$ to obtain
\begin{equation} \label{bww}
\begin{array}{l}
\vspace{2mm}\displaystyle \left(BW,W\right)_X\\=
\vspace{2mm}\displaystyle\frac{1}{b_1}\left[-b_2\int_{\RR^+}p_{xxx}pdx+3bb_2\int_{\RR^+}p_{xx}pdx-3b^2b_2\int_{\RR^+}p_{x}pdx+b^3b_2\int_{\RR^+}p^2dx\right.\\
\vspace{2mm}\displaystyle\left.-a_3b_2\int_{\RR^+}q_{xxx}pdx+3ba_3b_2\int_{\RR^+}q_{xx}pdx-3b^2a_3b_2\int_{\RR^+}q_xpdx+b^3a_3b_2\int_{\RR^+}pqdx \right. \\
\vspace{2mm}\displaystyle\left.-\int_{\mathbb{R}^+} a(x)p^2dx\right]-\frac{r}{b_1}\int_{\RR^+}pq_xdx+\frac{br}{b_1}\int_{\RR^+}q^2dx-\frac{1}{b_1}\int_{\mathbb{R}^+} a(x)q^2dx\\
\vspace{2mm}\displaystyle-\frac{a_3b_2}{b_1}\int_{\RR^+}(p_{xxx}-3bp_{xx}+3b^2p_x-b^3p)qdx-\frac{1}{b_1}\int_{\RR^+}(q_{xxx}-3bq_{xx}+3b^2q_x-b^3q)qdx.\\
\end{array}
\end{equation}
We observe that
\begin{equation}
\begin{array}{l}
\vspace{1mm}\displaystyle-b_2\int_{\RR^+}p_{xxx}pdx = -\frac{b_2}{2}p_x^2(0,t) \\
\vspace{1mm}\displaystyle3bb_2\int_{\RR^+}p_{xx}pdx = -3bb_2\int_{\RR^+}p_x^2dx \\
\vspace{1mm}\displaystyle-3b^2b_2\int_{\RR^+}p_{x}pdx = 0\\
\vspace{1mm}\displaystyle3ba_3b_2\int_{\RR^+}q_{xx}pdx = -3ba_3b_2\int_{\RR^+}q_xp_xdx\\
\vspace{1mm}\displaystyle-3b^2a_3b_2\int_{\RR^+}q_xpdx = 3b^2a_3b_2\int_{\RR^+}qp_xdx\\
\vspace{1mm}\displaystyle-a_3b_2\int_{\RR^+}p_{xxx}qdx = a_3b_2\int_{\RR^+}q_{xxx}pdx+a_3b_2p_x(0,t)q_x(0,t)\\
\vspace{1mm}\displaystyle3ba_3b_2\int_{\RR^+}p_{xx}qdx = -3ba_3b_2\int_{\RR^+}p_xq_xdx.\nonumber
\end{array}
\end{equation}
Similar computations can be done for the remaining terms. Combining \eqref{bww} and the above estimates, the following holds
\begin{equation}\label{bww2}
\begin{array}[c]{l}
\vspace{2mm}\displaystyle \left(BW,W\right)_X\\ \leq
\vspace{2mm}\displaystyle\frac{1}{b_1}\left[-3b\int_{\RR^+}\left(b_2p_x^2+q_x^2\right)dx-6ba_3b_2\int_{\RR^+}p_xq_xdx+(b^3+br)\int_{\RR^+}\left(b_2p^2+q^2\right)dx \right.\\
\vspace{2mm}\displaystyle\left.+\left|2b^3a_3b_2\right|\int_{\RR^+}pqdx-r\int_{\RR^+}pq_xdx-\frac{1}{2}\left(b_2p_x^2(0,t)+q_x^2(0,t)\right)+a_3b_2q_x(0,t)p_x(0,t)\right].\\
\end{array}
\end{equation}
In order to conclude this step of the proof we estimate terms on the right hand side of \eqref{bww2} using H\"{o}lder inequality:
\begin{equation}
\begin{array}{l}
\vspace{2mm}\left|6ba_3b_2\displaystyle\int_{\mathbb{R}^+} p_xq_xdx\right|\ \leq\ 3b\left(\displaystyle\int_{\mathbb{R}^+} b_2p_x^2dx+a_3^2b_2\int_{\mathbb{R}^+} q_x^2dx\right)\\
\vspace{2mm}\displaystyle\left|2b^3a_3b_2\displaystyle\int_{\mathbb{R}^+} pqdx\right|\ \leq\ (b^3\left|a_3\right|b_1)\frac{b_2}{b_1}\displaystyle\int_{\mathbb{R}^+} p^2dx+(b^3\left|a_3\right|b_2)\displaystyle\int_{\mathbb{R}^+} q^2dx\\
\vspace{2mm}\qquad\qquad\qquad\qquad\,\,\displaystyle\ \leq\ [(b_1+b_2)\left|a_3\right|b^3]\left(\frac{b_2}{b_1}\int_{\mathbb{R}^+} p^2dx+\int_{\mathbb{R}^+} q^2dx\right)\\
\vspace{2mm}\displaystyle \left|-r\int_{\mathbb{R}^+} pq_xdx\right|\ \leq\ \frac{r^2}{2}\int_{\mathbb{R}^+} q_x^2dx+\left(\frac{b_1}{2b_2}\right)\frac{b_2}{b_1}\int_{\mathbb{R}^+} p^2dx\\
\displaystyle a_3\sqrt{b_2}q_x(0,t)\sqrt{b_2}p_x(0,t)\ \leq\ \frac{a_3^2b_2}{2}q_x^2(0,t)+\frac{b_2}{2}p_x^2(0,t).\nonumber
\end{array}
\end{equation}
Recalling that $a^2_3b_2<1$ and $r$ is sufficiently small (in particular, $\frac{r^2}{2}-3b(1-a_3^2b_2)<0$), from the above estimates and \eqref{bww2}, we get
\begin{equation}\label{b2}
\begin{array}{l}
\vspace{2mm}\displaystyle \left(BW,W\right)_X\\
\vspace{2mm}\displaystyle\leq \frac{1}{b_1}\left\{\frac{r^2}{2}-3b\left(1-a_3^2b_2\right)\int_{\RR^+}q_x^2dx+\left[b_1\left(b^3+br\right)+b^3\left|a_3\right|(b_1+b_2)+\frac{b_1}{2b_2}\right]\frac{b_2}{b_1}\int_{\mathbb{R}^+} p^2dx\right.\\
\vspace{2mm}\displaystyle\left.+\left[b^3+br+b^3\left|a_3\right|(b_1+b_2)\right]\int_{\mathbb{R}^+} q^2dx-\frac{1}{2}\left(1-a_3^2b_2\right)q_x^2(0,t)\right\}\\
\vspace{2mm}\displaystyle \leq \frac{1}{b_1}\left[\left(b_1+1\right)\left(b^3+br\right)+b^3\left|a_3\right|(b_1+b_2)+\frac{b_1}{2b_2}\right]	 \left(\frac{b_2}{b_1}\int_{\mathbb{R}^+} p^2dx+\int_{\mathbb{R}^+} q^2dx\right).
\end{array}
\end{equation}
Then, choosing
 $$\lambda=\frac{1}{b_1}\left[\left(b_1+1\right)\left(b^3+br\right)+b^3\left|a_3\right|(b_1+b_2)+\frac{b_1}{2b_2}\right]>0$$
we deduce that
$$\left((B-\lambda)W,W\right)_X\leq0.$$
Hence, $B-\lambda$ is a dissipative in $X$. Now, picking $W\in D(B^*)$ and performing as in the previous computation, we can deduce that
$$
\displaystyle\left(W,B^*W\right)_X= \left(BW,W\right)_X \leq \lambda\left(W,W\right)_X,
$$
which allows us to conclude that $B^*-\lambda$ is dissipative. The proof is now complete.
\end{proof}

The following existence and uniqueness result for the solution of the system \eqref{sis lin} is a
direct consequence of Lemma \ref{exist linear} and the semigroups theory.
\begin{theorem}\label{t1}
Let $U_0 \in [L^2_b(\RR^+)]^2$. Then, there exists a unique weak solution $S(t)U_0$ of \eqref{sis lin} such that $U\in \mathcal{C}(\RR^+;[L^2_b(\RR^+)]^2)$. If $U_0\in D(A)$, system \eqref{sis lin} has a unique (classical) solution $U\in \mathcal{C}(\RR^+;D(A))\cap \mathcal{C}^1(\RR^+;[L^2_b(\RR^+)]^2)$.
\end{theorem}

An additional regularity result for the weak solutions of \eqref{sis lin} is proved in the
next Lemma. It is essentially the same result proved in Lemma \ref{lema 1}.

\begin{lemma}\label{lema 0}
Let $U_0 \in \left[L^2_b(\RR^+)\right]^2$ and $U$ the weak solution of \eqref{sis lin}. Then, for any  $T>0$, $U\in L^2(0,T;[H^1(\RR^+)]^2)$
and
\begin{equation}\label{12'}
\left\|U\right\|_{L^\infty(0,T;[L^2_b]^2)} + \left\|U_x\right\|_{L^2(0,T;[L^2_b]^2)}\leq C\left\|U_0\right\|_{[L^2_b]^2} ,
\end{equation}
where $C=C(T)$ is a positive constant.
\end{lemma}
\begin{proof}
The proof is obtained following the arguments developed in the proof of Lemma \ref{lema 1}, therefore we omit the details.
\end{proof}

\subsection{The nonlinear system}

In order to make the reading easier we split the proof of the main result (Theorem
\ref{nl}) into several steps. The first result is taken from \cite{pazoto-rosier}. For the sake of completeness
we repeat it here.

\begin{lemma}\label{desig b}
If $U\in [H^1_b(\RR^+)]^2$, then
\begin{equation}
 \left\|U(x)e^{bx}\right\|_{[L^\infty(\RR^+)]^2}\leq \sqrt{2+2b}\left\|U\right\|^{\frac{1}{2}}_{[L^2_b(\RR^+)]^2}\left\|U\right\|^{\frac{1}{2}}_{[H^1_b(\RR^+)]^2}.
\end{equation}
\end{lemma}
\begin{proof}
Let $\xi\in \RR^+$ and $u, v\in H^1_b(\RR^+)$. From Cauchy-Schwarz inequality, we get
\begin{equation}
\begin{array}{l}
\displaystyle u^2(\xi)e^{2b\xi}=\int_0^\xi \left(u^2e^{2bx}\right)_xdx = \int_0^\xi \left(2uu_xe^{2bx}+2bu^2e^{2bx}\right)dx\\
\vspace{3mm}\displaystyle\leq 2\left(\int_{\mathbb{R}^+} u^2e^{2bx}dx\right)^\frac{1}{2}\left(\int_{\mathbb{R}^+} u_x^2e^{2bx}dx\right)^\frac{1}{2}+2b\int_{\mathbb{R}^+} u^2e^{2bx}dx\\
\vspace{3mm}\displaystyle\leq (2+2b)\left\|u\right\|_{L^2_b(\RR^+)}\left\|u\right\|_{H^1_b(\RR^+)}.\nonumber
\end{array}
\end{equation}
Consequently,
$$
\begin{array}{l}
\displaystyle \left(u^2(\xi)+v^2(\xi)\right)e^{2b\xi}\leq (2+2b)\left(\left\|u\right\|_{L^2_b(\RR^+)}\left\|u\right\|_{H^1_b(\RR^+)}+\left\|v\right\|_{L^2_b(\RR^+)}\left\|v\right\|_{H^1_b(\RR^+)}\right)\\
\vspace{3mm}\qquad\qquad\qquad\qquad\,\,\,\leq(2+2b)\left\|U\right\|_{[L^2_b(\RR^+)]^2}\left\|U\right\|_{[H^1_b(\RR^+)]^2},
\end{array}
$$
and the result follows.
\end{proof}

The main result of this section reads as follows:

\begin{theorem}\label{nl} For any $U_0\in \left[L^2_b(\RR^+)\right]^2$ and $T>0$, there exists a unique solution
$U\in \mathcal{C}([0,T];\left[L^2_b(\RR^+)\right]^2)\cap L^2 (0,T;\left[H^1_b(\RR^+)\right]^2 )$ of \eqref{e1}.
\end{theorem}
\begin{proof} First, we consider the nonhomogeneous linear problem
\begin{equation}\label{sis lin 2'}
\left\{
\begin{array}{l}%
b_2u_{t}+b_2u_{xxx}+a_{3}b_2v_{xxx}+a(x)u=f_1\\
b_{1}v_{t}+v_{xxx}+b_{2}a_{3}u_{xxx}+rv_{x}+a(x)v=f_2,\ \ x\in\RR^+, t\in(0,T),    \\
u(0,t)=v(0,t)=0,\ \ t\in(0,T),\\
u(x,0)=u_{0}(x), v(x,0)=v_{0}(x),\ \  x\in\RR^+.
\end{array}
\right.
\end{equation}
For any $f_1, f_2\in L^1\left([0,T];L^2_b(\RR^+)\right)$ and $u_0, v_0 \in L^2_b(\RR^+)$, Lemma \eqref{exist linear}
guarantees that \eqref{sis lin 2'} has a unique weak solution, such that $u, v \in \mathcal{C}([0,T]; L^2_b(\RR^+))$. Moreover, performing as in Lemma \ref{lema 1} we deduce that
\begin{equation} \label{13'}
\begin{array}{l}
\vspace{3mm}\displaystyle \sup_{0\leq t\leq T}\left(\left\|u(t)\right\|_{L^2_b}+\left\|v(t)\right\|_{L^2_b}\right)+\left(\int_0^T\int_{\mathbb{R}^+} \left(u_x^2+v_x^2\right)e^{2bx}dxdt\right)^\frac{1}{2}\\
\vspace{3mm}\displaystyle \leq C\left(\left\|u_0\right\|_{L^2_b}+\left\|v_0\right\|_{L^2_b}+\int_{0}^T\left(\left\|f_1(t)\right\|_{L^2_b}+\left\|f_2(t)\right\|_{L^2_b}\right)dt\right),
\end{array}
\end{equation}
for some positive constant $C=C(T)$ nondecreasing in $T$.

Now we turn our attention to the local well-posedness of \eqref{e1} in the space
$$F=\mathcal{C}([0,T];\left[L^2_b(\RR^+)\right]^2)\cap L^2(0,T;\left[H^1_b(\RR^+)\right]^2)$$
endowed with its natural norm. We introduce the map $\Gamma$ defined by
\begin{equation}
\Gamma(U)(t)=S(t)U_0+\int_0^t S(t-s)N(U)(s)ds
\end{equation}
where $\{S(t)\}_{t\geq 0}$ is the semigroup associated to the linear problem and
\begin{equation}
N(U)=\left(
\begin{array}{c}
\vspace{3mm}\displaystyle -uu_x-a_1vv_x-a_2(uv)_x\\
\vspace{3mm}\displaystyle -\dfrac{1}{b_1}vv_x-\dfrac{a_2b_2}{b_1}uu_x-\dfrac{a_1b_2}{b_1}(uv)_x
\end{array}
\right).
\end{equation}
We shall prove that $\Gamma$ has a fixed point in some ball $B_R(0)$ of the space $F$. Therefore, the following result is needed:
\begin{lemma}\label{KT}
There exists a constant $K>0$, such that
$$\left\|\Gamma(U_2)-\Gamma(U_1)\right\|_F \leq KT^\frac{1}{4}\left( \left\|U_1\right\|_F+\left\|U_2\right\|_F   \right)\left\|U_2-U_1\right\|_F,$$ $\forall\  U_1, U_2\in F$.
\end{lemma}
\begin{proof}
Let $U_1=\left(
\begin{array}[c]{l}
u_1\\
v_1
\end{array}
\right)$, $U_2=\left(
 \begin{array}[c]{l}
 u_{2}\\
 v_{2}
 \end{array}
 \right)\in F$ and $T>0$ to be chosen later. Then, using \eqref{13'} we obtain a positive constant $C>0$, such that
\begin{equation}\label{gama}
\begin{array}[c]{l}
\vspace{1mm} \displaystyle \left\|\Gamma(U_2)-\Gamma(U_1)\right\|_F\\
\vspace{1mm}\displaystyle\leq CL\left(\left\|u_2u_{2x}-u_1u_{1x}\right\|_{L^1(0,T;L^2_b)}+\left\|v_2v_{2x}-v_1v_{1x}\right\|_{L^1(0,T;L^2_b)}
+\left\|u_{2x}v_2-u_{1x}v_1\right\|_{L^1(0,T;L^2_b)}\right.\\
\displaystyle \left.+\left\|u_2v_{2x}-u_1v_{1x}\right\|_{L^1(0,T;L^2_b)}\right),
\end{array}
\end{equation}
where $L=1+\left|a_1\right|+\left|a_2\right|+\dfrac{1}{b_1}+\dfrac{b_2}{b_1}\left(\left|a_1\right|+\left|a_2\right|\right)$. To estimate the terms on the right hand side of the above inequality, we first combine the triangular inequality and H\"{o}lder inequality to obtain
\begin{equation}\label{bor 1}
\begin{array}{l}
\vspace{1mm} \displaystyle \left\|u_2u_{2x}-u_1u_{1x}\right\|_{L^1(0,T;L^2_b)}\  =\left\|\left(u_2-u_1\right)u_{1x}+u_2\left(u_{2x}-u_{1x}\right)\right\|_{L^1(0,T;L^2_b)}\\ \vspace{3mm}\displaystyle\leq\int_{0}^T \left\|u_2-u_1\right\|_{L^\infty}\left\|u_{1x}\right\|_{L^2_b}dt+\int_{0}^T \left\|u_2\right\|_{L^\infty}\left\|u_{2x}-u_{1x}\right\|_{L^2_b}dt\\
\vspace{1mm}\displaystyle=\left\|u_2-u_1\right\|_{L^2(0,T;L^\infty)}\left\|u_{1x}\right\|_{L^2(0,T;L^2_b)}+ \left\|u_{2x}-u_{1x}\right\|_{L^2(0,T;L^2_b)}\left\|u_2\right\|_{L^2(0,T;L^\infty)}.
\end{array}
\end{equation}
Analogously,
\begin{equation}\label{bor 2}
\begin{array}{l}
\vspace{1mm} \displaystyle \left\|u_{2x}v_2-u_{1x}v_1\right\|_{L^1(0,T;L^2_b)}
=\left\|u_{2x}\left(v_2-v_1\right)+v_1\left(u_{2x}-u_{1x}\right)\right\|_{L^1(0,T;L^2_b)}\\
\leq \left\|u_{2x}\right\|_{L^2(0,T;L^2_b)}\left\|v_2-v_1\right\|_{L^2(0,T;L^\infty)}+ \left\|u_{2x}-u_{1x}\right\|_{L^2(0,T;L^2_b)}\left\|v_1\right\|_{L^2(0,T;L^\infty)}.
\end{array}
\end{equation}
Similar conclusions remains valid for the other terms. Now, we recall that, from Gagliardo-Niremberg inequality, we obtain a positive constant $C_{GN}>0$, such that the following inequality holds
\begin{equation}\label{bor 3}
\begin{array}{l}
\vspace{1mm}\displaystyle \left\|w\right\|_{L^2(0,T;L^\infty)}=\left(\int_{0}^T \left\|w\right\|^2_{L^\infty}dt\right)^\frac{1}{2} \leq
\left(\int_{0}^T C^2_{GN}\left\|w\right\|_{L^2}\left\|u_{1x}\right\|_{L^2}dt\right)^\frac{1}{2}\\
\vspace{1mm}\displaystyle\leq C_{GN}\left(\int_{0}^T \left\|w\right\|_{L^2_b}\left\|u_{1x}\right\|_{L^2_b}dt\right)^\frac{1}{2} \leq  C_{GN}\left[ \left(\int_{0}^T \left\|w\right\|^2_{L^2_b}dt\right)^\frac{1}{2}\left(\int_{0}^T\left\|u_{1x}\right\|^2_{L^2_b} dt\right)^\frac{1}{2}\right]^\frac{1}{2}\\
\vspace{1mm}\displaystyle\leq C_{GN}T^\frac{1}{4}\left\|w\right\|_{L^\infty(0,T;L^2_b)}^\frac{1}{2}\left\|u_{1x}\right\|_{L^2(0,T;L^2_b)}^\frac{1}{2}.
\end{array}
\end{equation}
Therefore, returning to \eqref{gama} and taking \eqref{bor 1}-\eqref{bor 3} into account, we get
\begin{equation}\label{contrai}
\begin{array}{l}
\vspace{3mm} \displaystyle \left\|\Gamma(U_2)-\Gamma(U_1)\right\|_F \\
\leq CLC_{GN}T^{\frac{1}{4}}\left[\left(\left\|u_{1x}\right\|_{L^2(0,T;L^2_b)}\left\|u_2-u_1\right\|_{L^\infty(0,T;L_b^2)}^\frac{1}{2}\left\|u_{2x}-u_{1x}\right\|_{L^2(0,T;L^2_b)}^\frac{1}{2}\right)\right.\\
\vspace{3mm}\displaystyle + \left.\left(\left\|u_2\right\|^\frac{1}{2}_{L^\infty(0,T;L_b^2)}\left\|u_{2x}\right\|_{L^2(0,T;L^2_b)}^\frac{1}{2}\left\|u_{2x}-u_{1x}\right\|_{L^2(0,T;L^2_b)}\right)\right.\\
\vspace{3mm}\displaystyle+\left.\left(\left\|v_{1x}\right\|_{L^2(0,T;L^2_b)}\left\|v_2-v_1\right\|^\frac{1}{2}_{L^\infty(0,T;L_b^2)}\left\|v_{2x}-v_{1x}\right\|^\frac{1}{2}_{L^2(0,T;L^2_b)}\right)\right.\\
\vspace{3mm}\displaystyle+\left.\left(\left\|v_2\right\|^\frac{1}{2}_{L^\infty(0,T;L_b^2)}\left\|v_{2x}\right\|_{L^2(0,T;L^2_b)}^\frac{1}{2}\left\|v_{2x}-v_{1x}\right\|_{L^2(0,T;L^2_b)}\right)\right.\\
\vspace{3mm}\displaystyle+  \left(\left\|u_{2x}\right\|_{L^2(0,T;L^2_b)}\left\|v_2-v_1\right\|^\frac{1}{2}_{L^\infty(0,T;L_b^2)}\left\|v_{2x}-v_{1x}\right\|_{L^2(0,T;L^2_b)}^\frac{1}{2}\right) \\
\vspace{3mm}\displaystyle +\left. \left(\left\|v_1\right\|_{L^\infty(0,T;L_b^2)}^\frac{1}{2}\left\|v_1\right\|_{L^2(0,T;L^2_b)}^\frac{1}{2}\left\|u_{2x}-u_{1x}\right\|_{L^2(0,T;L^2_b)} \right) \right.\\
\vspace{3mm}\displaystyle \left. +\left(\left\|v_{2x}\right\|_{L^2(0,T;L^2_b)}\left\|u_2-u_1\right\|_{L^\infty(0,T;L_b^2)}^\frac{1}{2}\left\|u_{2x}-u_{1x}\right\|_{L^2(0,T;L^2_b)}^\frac{1}{2}\right) \right.\\
\vspace{3mm}\displaystyle+\left.\left(\left\|u_1\right\|_{L^\infty(0,T;L_b^2)}^\frac{1}{2}\left\|u_{1x}\right\|_{L^2(0,T;L^2_b)}^\frac{1}{2}\left\|v_{2x}-v_{1x}\right\|_{L^2(0,T;L^2_b)}\right)\right]\\
\vspace{3mm}\displaystyle \leq \dfrac{3CLC_{GN}T^\frac{1}{4}}{2}\left(\left\|u_{1x}\right\|_{L^2(0,T;L^2_b)}+\left\|u_{2x}\right\|_{L^2(0,T;L^2_b)}+\left\|v_{1x}\right\|_{L^2(0,T;L^2_b)}+\left\|v_{2x}\right\|_{L^2(0,T;L^2_b)}\right.\\
\vspace{3mm}\displaystyle +\left.\left\|u_1\right\|_{L^\infty(0,T;L^2)}+\left\|u_2\right\|_{L^\infty(0,T;L^2)}+\left\|v_1\right\|_{L^\infty(0,T;L^2)}+\left\|v_2\right\|_{L^\infty(0,T;L^2)}\right)\times\\
\vspace{3mm}\displaystyle\times \left(\left\|u_{2x}-u_{1x}\right\|_{L^2(0,T;L^2_b)}+\left\|v_{2x}-v_{1x}\right\|_{L^2(0,T;L^2_b)}+\left\|u_2-u_1\right\|_{L^\infty(0,T;L^2)}+\left\|v_2-v_1\right\|_{L^\infty(0,T;L^2)}\right)\\
\vspace{3mm}\displaystyle \leq KT^\frac{1}{4}\left( \left\|U_2\right\|_F + \left\|U_1\right\|_F \right) \left\|U_2-U_1\right\|_F.
\end{array}
\end{equation}
The proof of Lemma \ref{KT} is complete.
\end{proof}
The above discussion allows us to conclude the local well-posedness result. Indeed, Lemma \ref{lema 0} and Lemma \ref{KT} show that $\Gamma U\in F$  and
$$\left\|\Gamma(U)\right\|_F\leq C\left(\left\|U_0\right\|_{[L^2_b]^2}+T^\frac{1}{4}\left\|U\right\|^2_F\right).$$
The above estimate guarantee that $\Gamma$ maps $B_R(0)\subset F$ into itself if we choose $R=2C\left\|U_0\right\|_{[L^2_b]^2}$ and $T>0$ sufficiently small. Moreover, for this choice of $R$ and $T$ small enough, Lemma \ref{KT} allows us to conclude that $\Gamma$ is a contraction in $F$. Then, by the contraction mapping theorem, we obtain the local existence and uniqueness result. The global existence comes from the a priori bound of the solutions established in the following Lemmas.

\begin{lemma}\label{lema 1}
Let $U_0 \in \left[L^2_b(\RR^+)\right]^2$ and $U$ the solution of \eqref{e1}. Then, for any  $T>0$
\begin{equation}\label{10'}
\begin{array}{l}
\vspace{3mm}\displaystyle\frac{1}{2}\int_{\RR^+} \left(b_2u^2(x,T)+b_1v^2(x,T)-b_2u_0^2(x)-b_1v_0^2(x)dx\right)dx \\
\vspace{3mm}\displaystyle+\frac{1}{2}\int_{0}^T \left(\sqrt{b_2}u_x(0,t)+\sqrt{a_3^2b_2}v_x(0,t)\right)^2dt+\frac{1}{2}(1-a_3^2b_2)\int_{0}^T v_x^2(0,t)dt\\
\vspace{3mm}\displaystyle+\int_0^T\int_{\mathbb{R}^+} a(x)\left(b_2u^2+v^2\right)dxdt\ =\ 0\
\end{array}
\end{equation}
and
\begin{equation}\label{11'}
\begin{array}{l}
\vspace{3mm}\displaystyle\frac{1}{2}\int_{\mathbb{R}^+} \left(b_2u^2(x,T)+b_1v^2(x,T)-b_2u_0^2(x)-b_1v_0^2(x)\right)e^{2bx}dx\\
\vspace{3mm}\displaystyle+6ba_3b_2\int_0^T\int_{\mathbb{R}^+} u_xv_x e^{2bx}dxdt+3b\int_0^T\int_{\mathbb{R}^+} \left(b_2u_x^2+v_x^2\right)e^{2bx}dxdt\\
\vspace{3mm}\displaystyle-4b^3\int_0^T\int_{\mathbb{R}^+} \left(b_2u^2+v^2\right)e^{2bx}dxdt-\frac{2b}{3}\int_0^T\int_{\mathbb{R}^+} \left(b_2u^3+v^3\right)e^{2bx}dxdt\\
\vspace{3mm}\displaystyle-8b^3a_3b_2\int_0^T\int_{\mathbb{R}^+} uve^{2bx}dxdt-2ba_1b_2\int_0^T\int_{\mathbb{R}^+} \left(uv^2+vu^2\right)e^{2bx}dxdt\\
\vspace{3mm}\displaystyle+\frac{1}{2}\int_{0}^T \left(\sqrt{b_2}u_x(0,t)+\sqrt{a_3^2b_2}v_x(0,t)\right)^2dt+\frac{1}{2}(1-a_3^2b_2)\int_{0}^T v_x^2(0,t)dt\\
\vspace{3mm}\displaystyle-br\int_0^T\int_{\mathbb{R}^+} v^2e^{2bx}dxdt+\int_0^T\int_{\mathbb{R}^+} a(x)\left(b_2u^2+v^2\right)e^{2bx}dxdt\ =\ 0\ .
\end{array}
\end{equation}
As a consequence,
\begin{equation}\label{12' b}
\left\|U\right\|_{L^\infty(0,T;[L^2_b]^2)} + \left\|U_x\right\|_{L^2(0,T;[L^2_b]^2)}\leq C\left\|U_0\right\|_{[L^2_b]^2} ,
\end{equation}
where $C=C(T)$ is a positive constant.
\end{lemma}
\begin{proof}
Identity \eqref{10'} is the classical energy dissipation law taken directly from \cite{pazoto-souza}. It is repeated here for the reader's convenience.

Identity \eqref{11'} is obtained multiplying the first equation in \eqref{e1} by $b_2\left(e^{2bx}-1\right)u$, the second one by
$\left(e^{2bx}-1\right)v$ and adding the resulting identities after integrating over $\RR^+\times\left(0,T\right)$. Therefore, the following identities will be needed:

\begin{equation}\label{i1}
b_2\int_0^T\int_{\mathbb{R}^+} (e^{2bx}-1)uu_tdxdt\ =\ \frac{b_2}{2}\int_{\mathbb{R}^+}(e^{2bx}-1)u^2(x,T)dx-\frac{b_2}{2}\int_{\mathbb{R}^+}(e^{2bx}-1)u_0^2(x)dx\
\end{equation}
\begin{equation}\label{i3}
b_2\int_0^T\int_{\mathbb{R}^+} (e^{2bx}-1)u_{xxx}udxdt\ =
\displaystyle 3bb_2\int_0^T\int_{\mathbb{R}^+} u_x^2e^{2bx}dxdt-4b^3b_2\int_0^T\int_{\mathbb{R}^+} u^2e^{2bx}dxdt
\end{equation}

\begin{equation}\label{i11}
\begin{array}{l}
\displaystyle r\int_0^T\int_{\mathbb{R}^+} v_xv(e^{2bx}-1)dxdt = -br\int_0^T\int_{\mathbb{R}^+} v^2e^{2bx}dxdt
\end{array}
\end{equation}

\begin{equation}\label{i9}
\begin{array}{l}
\vspace{3mm}\displaystyle a_3b_2\int_0^T\int_{\mathbb{R}^+} \left(e^{2bx}-1\right)u_{xxx}vdxdt+a_3b_2\int_0^T\int_{\mathbb{R}^+} \left(e^{2bx}-1\right)v_{xxx}udxdt\\
=-8b^3a_3b_2\displaystyle\int_0^T\int_{\mathbb{R}^+} uve^{2bx}dxdt+6ba_3b_2\int_0^T\int_{\mathbb{R}^+} u_xv_xe^{2bx}dxdt.
\end{array}
\end{equation}
Similar conclusions remains valid for the remaining terms $(e^{2bx}-1)vv_t$ and $(e^{2bx}-1)vv_{xxx}$.
Moreover, since $a_1=a_2$, if we put together all the nonlinear terms in \eqref{e1}, it follows that
\begin{equation}\label{i10}
\begin{array}{l}
\vspace{3mm}\displaystyle b_2\int_0^T\int_{\mathbb{R}^+} u^2u_x(e^{2bx}-1)dxdt+a_1b_2\int_0^T\int_{\mathbb{R}^+} (uv)v_x(e^{2bx}-1)dxdt\\
\vspace{3mm}\displaystyle + a_1b_2\int_0^T\int_{\mathbb{R}^+} u(uv)_x(e^{2bx}-1)dxdt +\int_0^T\int_{\mathbb{R}^+} v^2v_x(e^{2bx}-1)dxdt\\
\vspace{3mm}\displaystyle + a_1b_2\int_0^T\int_{\mathbb{R}^+} (uv)u_x(e^{2bx}-1)dxdt + a_1b_2\int_0^T\int_{\mathbb{R}^+} v(uv)_x(e^{2bx}-1)dxdt\\
\vspace{3mm}\displaystyle= \frac{b_2}{3}\int_0^T\int_{\mathbb{R}^+} \left(u^3\right)_x(e^{2bx}-1)dxdt+\frac{1}{3}\int_0^T\int_{\mathbb{R}^+} \left(v^3\right)_x(e^{2bx}-1)dxdt\\
\vspace{3mm}\displaystyle + a_1b_2\int_0^T\int_{\mathbb{R}^+} \left(u^2v\right)_x(e^{2bx}-1)dxdt +\displaystyle a_1b_2\int_0^T\int_{\mathbb{R}^+} \left(v^2u\right)_x(e^{2bx}-1)dxdt\\
\vspace{3mm}\displaystyle= -\frac{2b}{3}\int_0^T\int_{\mathbb{R}^+} (b_2u^3+v^3)e^{2bx}dxdt-2ba_1b_2\int_0^T\int_{\mathbb{R}^+} (u^2v+uv^2)e^{2bx}dxdt.
\end{array}
\end{equation}
Then, taking the identities \eqref{i1}-\eqref{i10} into account, we get
\begin{equation}\label{ident 2}
\begin{array}{l}
\vspace{3mm}\displaystyle\frac{1}{2}\int_{\mathbb{R}^+}(e^{2bx}-1)\left(b_2u^2(x,T)-b_2u_0^2(x)\right)dx+\frac{1}{2}\int_{\mathbb{R}^+}(e^{2bx}-1)\left(b_1v^2(x,T)-b_1v_0^2(x)\right)dx\\
\vspace{3mm}\displaystyle+3b\int_0^T\int_{\mathbb{R}^+} (b_2u_x^2+v_x^2)e^{2bx}dxdt-4b^3\int_0^T\int_{\mathbb{R}^+} (b_2u^2+v^2)e^{2bx}dxdt\\
\vspace{3mm}\displaystyle+6ba_3b_2\int_0^T\int_{\mathbb{R}^+} u_xv_xe^{2bx}dxdt-8b^3a_3b_2\int_0^T\int_{\mathbb{R}^+} uve^{2bx}dxdt\\
\vspace{3mm}\displaystyle-\frac{2b}{3}\int_0^T\int_{\mathbb{R}^+}(b_2u^3+v^3)e^{2bx}dxdt-2ba_1b_2\int_0^T\int_{\mathbb{R}^+}(u^2v+uv^2)e^{2bx}dxdt\\
\vspace{3mm}\displaystyle-br\int_0^T\int_{\mathbb{R}^+} v^2e^{2bx}dxdt+\int_0^T\int_{\mathbb{R}^+} a(x)(b_2u^2+v^2)(e^{2bx}-1)dxdt = 0.
\end{array}
\end{equation}
\\
Adding \eqref{10'} and \eqref{ident 2}, we obtain \eqref{11'}.

Estimate \eqref{12' b} is proved combining \eqref{11'} and Gronwall Lemma. Therefore, the next steps are devoted to estimate the terms in  \eqref{11'}. From H\"{o}lder inequality we obtain
\begin{equation}\label{est 1}
\begin{array}{l}
\vspace{3mm}\displaystyle\left|6ba_3b_2\int_0^T\int_{\mathbb{R}^+} u_xv_xe^{2bx}dxdt\right|\ = \left|\int_0^T\int_{\mathbb{R}^+} \left(\sqrt{3bb_2}u_x\right)\left(2a_3\sqrt{3bb_2}v_x\right)e^{2bx}dxdt\right| \\
\leq
\vspace{3mm}\displaystyle 3b\left(\frac{1}{2}\int_0^T\int_{\mathbb{R}^+} b_2u_x^2e^{2bx}dxdt + 2a_3^2b_2\int_0^T\int_{\mathbb{R}^+} v_x^2e^{2bx}dxdt\right)
\end{array}
\end{equation}
and
\begin{equation}\label{est 1.5}
\left|-8b^3a_3b_2\int_0^T\int_{\mathbb{R}^+} uve^{2bx}dxdt\right|\leq 4b^3\left(a^2_3b_2\int_0^T\int_{\mathbb{R}^+} v^2e^{2bx}dxdt+\int_0^T\int_{\mathbb{R}^+} b_2u^2e^{2bx}dxdt\right).
\end{equation}
In order to estimate the cubic terms we proceed as follows. First, observe that
\begin{equation}\label{estab -3}
\begin{array}{l}
\vspace{3mm}\displaystyle \left|\int_{\mathbb{R}^+} u^2e^{2bx}dx\right| = \left|-\dfrac{1}{b}\int_{\mathbb{R}^+} uu_xe^{2bx}dx\right|\leq \frac{1}{b}\left(\int_{\mathbb{R}^+} u^2e^{2bx}dx\right)^\frac{1}{2}\left(\int_{\mathbb{R}^+} u_x^2e^{2bx}dx\right)^\frac{1}{2},\nonumber
\end{array}
\end{equation}
therefore
\begin{equation}\label{estab -2}
\vspace{3mm}\displaystyle \int_{\mathbb{R}^+} u^2e^{2bx}dx\leq \dfrac{1}{b^2}\int_{\mathbb{R}^+} u_x^2e^{2bx}dx.
\end{equation}
Combining Lemma \ref{desig b} and the above inequality we obtain a positive constant $\tilde{c}$ satisfying
\begin{equation}\label{ipom}
\begin{array}{l}
\vspace{3mm}\displaystyle \|b_2u(x)e^{bx}\|_{L^\infty(\RR^+)}\leq \tilde{c}\,b_2\left\|u_x\right\|_{L^2_b}.
\end{array}
\end{equation}
From \eqref{ipom}, \eqref{10'} and H\"{o}lder inequality we deduce that
\begin{equation}\label{cubico 2}
\begin{array}{l}
\vspace{3mm}\displaystyle \dfrac{2b}{3}\int_0^T\int_{\mathbb{R}^+} b_2u^3e^{2bx}dxdt\leq \dfrac{2bb_2}{3}\int_{0}^T \|u(x)e^{bx}\|_{L^\infty(\RR^+)}\left(\int_{\mathbb{R}^+} u^2e^{bx}dx\right)dt\\
\vspace{3mm}\displaystyle\leq \dfrac{2bb_2}{3}\tilde{c}\int_{0}^T \left\|u_x\right\|_{L^2_b}\left\|u\right\|_{L^2_b}\left\|u\right\|_{L^2}dt \leq b_2\int_{0}^T  \left\|u_x\right\|_{L^2_b}\left(\dfrac{2b}{3}\tilde{c}\left\|u_0\right\|_{L^2}\left\|u\right\|_{L^2_b}\right)dt\\
\displaystyle\leq \delta b_2\left\|u_x\right\|^2_{L^2(0,T;L^2_b)}+c_{1\delta}b_2\left\|u\right\|^2_{L^2(0,T;L^2_b)},
\end{array}
\end{equation}
 for any $\delta>0$, where $c_{1\delta}=c_{1\delta}(\|U_0\|^2_{[L^2]^2})$ is a positive constant. 
Using \eqref{cubico 2} the terms that are cubic in $u, v$ in \eqref{11'} can be estimated as follows
\begin{equation}\label{cubico final}
\begin{array}{l}
\vspace{3mm}\displaystyle 2ba_1b_2\int_0^T\int_{\mathbb{R}^+} \left(u^2v+uv^2\right)e^{2bx}dxdt
\vspace{3mm}\displaystyle +\dfrac{2b}{3}\int_0^T\int_{\mathbb{R}^+} \left(b_2u^3+v^3\right)e^{2bx}dxdt\\
\vspace{3mm}\displaystyle\leq 2\delta\left(b_2\left\|u_x\right\|^2_{L^2(0,T;L^2_b)}+\left\|v_x\right\|^2_{L^2(0,T;L^2_b)}\right)+k_\delta \left(b_2\left\|u\right\|^2_{L^2(0,T;L^2_b)}+\left\|v\right\|^2_{L^2(0,T;L^2_b)}\right),
\end{array}
\end{equation}
for any $\delta > 0$, where $k_\delta$ is positive constant that depends on $\delta$ and $\|U_0\|^2_{[L^2]^2}$.

We can now conclude the result. Indeed, identity \eqref{11'} together with estimates \eqref{est 1}, \eqref{est 1.5}, \eqref{cubico final}
and H\"{o}lder inequality lead to
\begin{equation}\label{est indep}
\begin{array}{l}
\vspace{3mm}\displaystyle\frac{1}{2}\int_{\mathbb{R}^+} \left(b_2u^2+b_1v^2\right)e^{2bx}dx+\frac{1}{2}\int_{\mathbb{R}^+} \left(b_2u_0^2(x)+b_1v_0^2(x)\right)e^{2bx}dx\\
\vspace{3mm}\displaystyle\leq \left[4b^3+4b^3\left(a_3^2b_2+1\right)+br+K_{\delta}\right]\int_0^T\int_{\mathbb{R}^+} \left(b_2u^2+v^2\right)e^{2bx}dxdt\\
\displaystyle+2\delta \int_0^T\int_{\mathbb{R}^+} \left(b_2u_x^2+v_x^2\right)e^{2bx}dxdt+\dfrac{3b}{2}\int_0^T\int_{\mathbb{R}^+} b_2u_x^2e^{2bx}dxdt+6ba_3^2b_2\int_0^T\int_{\mathbb{R}^+} v_x^2e^{2bx}dxdt.
\end{array}
\end{equation}
Then,
\begin{equation}\label{est indep 2}
\begin{array}{l}
\vspace{3mm}\displaystyle\frac{1}{2}\int_{\mathbb{R}^+} \left(b_2u^2+b_1v^2\right)e^{2bx}dx+\left[\frac{3b}{2}-2\delta\right]\int_0^T\int_{\mathbb{R}^+} b_2u_x^2e^{2bx}dxdt\\
\vspace{3mm}\displaystyle+\left[3b\left(1-2a_3^2b_2\right)-2\delta\right]\int_0^T\int_{\mathbb{R}^+} v_x^2e^{2bx}dxdt\leq \displaystyle\frac{1}{2}\int_{\mathbb{R}^+} \left(b_2u_0^2(x)+b_1v_0^2(x)\right)e^{2bx}dx\\
+\left[4b^3+4b^3\left(a_3^2b_2+1\right)+br+k_{\delta}\right]\displaystyle\int_0^T\int_{\mathbb{R}^+} \left(b_2u^2+v^2\right)e^{2bx}dxdt.
\end{array}
\end{equation}
Choosing $\delta<\min\left\{\dfrac{3b}{4},\dfrac{3b\left(1-2a_3^2b_2\right)}{2}\right\}$ we deduce that
\begin{equation}\label{est indep 3}
\begin{array}{l}
\vspace{3mm}\displaystyle\frac{1}{2}\int_{\mathbb{R}^+} \left(b_2u^2(x,T)+b_1v^2(x,T)\right)e^{2bx}dx \\
\displaystyle\leq \frac{\left(b_1+b_2\right)}{2}\left(\left\|u_0\right\|_{L^2_b}+\left\|v_0\right\|_{L^2_b}\right)+\gamma\int_0^T\int_{\mathbb{R}^+} \left(b_2u^2+b_1v^2\right)e^{2bx}dxdt,
\end{array}
\end{equation}
where $\gamma=\gamma\left(||U_0||_{[L^2]^2}\right)>0$. Then, by Gronwall Lemma we have the following bound for the solution
\begin{equation}\label{est indep 4}
\begin{array}{l}
\vspace{3mm}\displaystyle \left\|u(T)\right\|^2_{L^2_b}+\left\|v(T)\right\|^2_{L^2_b}\leq \left(b_1+b_2\right)\left(\left\|u_0\right\|^2_{L^2_b}+\left\|v_0\right\|^2_{L^2_b}\right)e^{2\gamma
T},\;\;\;\forall \;T>0.
\end{array}
\end{equation}
Consequently,
\begin{equation}\label{fim 1}
\left\|U\right\|_{L^\infty(0,T;[L^2_b]^2)}\leq C\left\|U_0\right\|_{[L^2_b]^2},
\end{equation}
where $C=C(T)>0$.  Now, combining \eqref{fim 1} and \eqref{est indep 2} we derive the estimate
\begin{equation}\label{fim 2}
\left\|U_x\right\|_{L^2(0,T;[L^2_b]^2)}\leq C\left\|U_0\right\|_{[L^2_b]^2},
\end{equation}
for some constant $C>0$. Finally, from \eqref{fim 1} and \eqref{fim 2} we obtain \eqref{12' b}.
\end{proof}
Lemma \ref{lema 1} guarantees that the solution exists globally in time. Uniqueness is standard and follows from Gronwall Lemma.
Now the proof of Theorem \ref{nl} is complete.
\end{proof}

\section{Decay}
\setcounter{equation}{0}

The main result of this section is obtained constructing a convenient Lyapunov function which is shown to decrease along trajectories. The proof combines the energy identities proved in the previous section and the result on the exponential decay for \eqref{e1} proved by the authors in \cite{pazoto-souza}. We observe that, since the total energy $E(t)$ is decreasing, their result can be summarized as follows.
\begin{theorem}\label{main-dec} Let $a=a(x)$ be any damping function satisfying \eqref{a}. Then, if $U_0=\left(
\begin{array}[c]{l}
u_0\\
v_0
\end{array}
\right)\in [L^2(\mathbb{R})]^2$ system \eqref{e1} is locally uniformly exponential stable, that is, for any $R>0$, there
exists a positive constant $C=C(R)$ satisfying
\begin{equation}\label{equiv}
\begin{array}{l}
\vspace{1mm}E(0)\leq C \displaystyle\int_0^T\Big[\int_\mathbb{R^+} a(x)(b_2u^2 + v^2) dx +
\displaystyle\frac{1}{2}\left(\sqrt{b_{2}}u_{x}(0,t)+\sqrt{a^{2}_{3}b_{2}}v_{x}(0,t)\right)^{2}\\
\qquad\,\,    +\displaystyle\frac{1}{2}\left(1-a^{2}_{3}b_{2}\right)v^{2}_{x}(0,t)\Big]dt,\nonumber
\end{array}
\end{equation}
provided $E(0)\leq R$.\nonumber
\end{theorem}

Indeed, if the above inequality holds, from \eqref{dissipation} we
have that $E(T)\leq \gamma E(0)$ with $0< \gamma <1$, which combined with the semigroup property allow us to derive the exponential decay of $E(t)$.

\vglue 0.2 cm

Now we state our main result. In order to make the reading easier, we split the proof in several steps.

\begin{theorem}\label{teo-dec}
Let $U$
\,be the solution of \eqref{e1} given by Theorem \ref{nl}. Then, for any $R>0$,  there exist constants $C>0$ and $\eta>0$, such that
\begin{equation}\label{alvo}
\displaystyle \left\|U(t)\right\|_{[L_b^2]^2}\leq Ce^{-\eta t}\left\|U_0\right\|_{[L_b^2]^2},\;\;\;\forall t\geq0,
\end{equation}
whenever $\left\|U_0\right\|_{[L_b^2]^2}\leq R$.
\end{theorem}
\begin{proof}
In order to obtain the result, we consider the Lyapunov function
\begin{equation}\label{Lyapunov func}
\mathcal{L}(U)(t) = \frac{1}{2}\int_{\mathbb{R}^+} \left(b_2u^2+b_1v^2\right)e^{2bx}dx+c_b\int_{\mathbb{R}^+} \left(b_2u^2+b_1v^2\right)dx,
\end{equation}
where $c_b > 0$ is constant to be chosen later. Then, to obtain \eqref{alvo} it is sufficient to prove that
\begin{equation}\label{semig-fin}
\mathcal{L}(U)-\mathcal{L}(U_0) \leq -c\mathcal{L}(U_0),
\end{equation}
for some positive constant $c$.

We multiply \eqref{10'} by $2c_b$ and add the resulting identity to \eqref{11'} to arrive at the relation
\begin{equation}\label{estab1}
\begin{array}{l}
\vspace{3mm}\displaystyle \mathcal{L}(U)-\mathcal{L}(U_0)\\
\vspace{3mm}\displaystyle =-c_b\int_{0}^T \left(\sqrt{b_2}u_x(0,t)+\sqrt{a_3^2b_2}v_x(0,t)\right)^2dt-c_b(1-a_3^2b_2)\int_{0}^T v_x^2(0,t)dt\\
\vspace{3mm}\displaystyle - \int_0^T\int_{\mathbb{R}^+} a(x)\left(b_2u^2+v^2\right)\left(c_b+e^{2bx}\right)dxdt - 6ba_3b_2\int_0^T\int_{\mathbb{R}^+} u_xv_xe^{2bx}dxdt\\
\vspace{3mm}\displaystyle - 3b\int_0^T\int_{\mathbb{R}^+} \left(b_2u_x^2+v_x^2\right)e^{2bx}dxdt+4b^3\int_0^T\int_{\mathbb{R}^+} \left(b_2u^2+v^2\right)e^{2bx}dxdt\\
\vspace{3mm}\displaystyle +\dfrac{2b}{3}\int_0^T\int_{\mathbb{R}^+} \left(
b_2u^3+v^3\right)e^{2bx}dxdt + 8b^3a_3b_2\int_0^T\int_{\mathbb{R}^+} uve^{2bx}dxdt\\
\vspace{3mm}\displaystyle +2ba_1b_2\int_0^T\int_{\mathbb{R}^+} \left(uv^2+vu^2\right)e^{2bx}dxdt-\dfrac{1}{2}\int_{0}^T \left(\sqrt{b_2}u_x(0,t)+\sqrt{a_3^2b_2}v_x(0,t)\right)^2dt \\
\displaystyle-\frac{1}{2}\left(1-a_3^2b_2\right)\int_{0}^T v_x^2(0,t)dt+br\int_0^T\int_{\mathbb{R}^+} v^2e^{2bx}dxdt.
\end{array}
\end{equation}
Then, from H\"{o}lder inequality, \eqref{est 1}, \eqref{est 1.5} and \eqref{cubico final} we obtain

\begin{equation}\label{estab2}
\begin{array}{l}
\vspace{3mm}\displaystyle \mathcal{L}(U)-\mathcal{L}(U_0)\\
\vspace{3mm}\displaystyle\leq -c_b\int_{0}^T \left(\sqrt{b_2}u_x(0,t)+\sqrt{a_3^2b_2}v_x(0,t)\right)^2dt -c_b\left(1-a_3^2b_2\right)\int_{0}^T v_x^2(0,t)dt\\
\vspace{3mm}\displaystyle -\int_0^T\int_{\mathbb{R}^+} a(x)\left(b_2u^2+v^2\right)\left(c_b+e^{2bx}\right)dxdt+\dfrac{3b}{2}\int_0^T\int_{\mathbb{R}^+} b_2u_x^2e^{2bx}dxdt\\
\vspace{3mm}\displaystyle+6ba_3^2b_2\int_0^T\int_{\mathbb{R}^+} v_x^2e^{2bx}dxdt+4b^3a_3^2b_2\int_0^T\int_{\mathbb{R}^+} v^2e^{2bx}dxdt+4b^3\int_0^T\int_{\mathbb{R}^+} b_2u^2e^{2bx}dxdt\\
\vspace{3mm}\displaystyle+k_{\delta}\int_0^T\int_{\mathbb{R}^+} \left(b_2u^2+v^2\right)e^{2bx}dxdt+2\delta\int_0^T\int_{\mathbb{R}^+} \left(b_2u_x^2+v_x^2\right)e^{2bx}dxdt\\
\vspace{3mm}\displaystyle+4b^3\int_0^T\int_{\mathbb{R}^+} \left(b_2u^2+v^2\right)e^{2bx}dxdt-3b\int_0^T\int_{\mathbb{R}^+} \left(b_2u_x^2+v_x^2\right)e^{2bx}dxdt\\
\vspace{3mm}\displaystyle-\dfrac{1}{2}\int_{0}^T \left(\sqrt{b_2}u_x(0,t)+\sqrt{a_3^2b_2}v_x(0,t)\right)^2dt-\dfrac{1}{2}\left(1-a_3^2b_2\right)\int_{0}^T v_x^2(0,t)dt\\
\vspace{3mm}\displaystyle+br\int_0^T\int_{\mathbb{R}^+} v^2e^{2bx}dxdt.
\end{array}
\end{equation}
Consequently, for any $\delta > 0$, \eqref{estab2} leads to the inequality
\begin{equation}\label{estab2.1}
\begin{array}{l}
\vspace{3mm}\displaystyle \mathcal{L}(U)-\mathcal{L}(U_0)\\
\vspace{3mm}\displaystyle\leq -\left(c_b+\frac{1}{2}\right)\int_{0}^T \left(\sqrt{b_2}u_x(0,t)+\sqrt{a_3^2b_2}v_x(0,t)\right)^2dt-\left(c_b+\frac{1}{2}\right)\left(1-a_3^2b_2\right)\int_{0}^T v_x^2(0,t)dt+\\
\vspace{3mm}\displaystyle-\int_0^T\int_{\mathbb{R}^+} a(x)\left(b_2u^2+v^2\right)\left(c_b+e^{2bx}\right)dxdt-\left(\dfrac{3b}{2}-2\delta\right)\int_0^T\int_{\mathbb{R}^+} b_2u_x^2e^{2bx}dxdt\\
\vspace{3mm}\displaystyle-\left(3b\left(1-2a_3^2b_2\right)-2\delta\right)\int_0^T\int_{\mathbb{R}^+} v_x^2e^{2bx}dxdt\\
\vspace{3mm}\displaystyle+\left[4b^3\left(1+a_3^2b_2\right)+br+k_{\delta}\right]\int_0^T\int_{\mathbb{R}^+} \left(b_2u^2+v^2\right)e^{2bx}dxdt,
\end{array}
\end{equation}
where $k_\delta$ is a positive constant that depends on $\delta$ and $\|U_0\|^2_{[L^2]^2}$. Now, we observe that $4b^3\left(1+a_3^2b_2\right)<6b^3$, therefore if we introduce the notation
\begin{equation}\label{estab2.5}
\rho_b=\left[6b^3+br+k_{\delta}\right],
\end{equation}\nonumber
from \eqref{estab2.1} we can deduce that
\begin{equation}\label{estab6}
\begin{array}{l}
\vspace{3mm}\displaystyle \mathcal{L}(U)-\mathcal{L}(U_0)\\
\vspace{3mm}\displaystyle\leq -\left(c_b+\frac{1}{2}\right)\int_{0}^T \left(\sqrt{b_2}u_x(0,t)+\sqrt{a_3^2b_2}v_x(0,t)\right)^2dt -\left(c_b+\frac{1}{2}\right)\left(1-a_3^2b_2\right)\int_{0}^T v_x^2(0,t)dt\\
\vspace{3mm}\displaystyle -\int_0^T\int_{\mathbb{R}^+} a(x)\left(b_2u^2+v^2\right)\left(c_b+e^{2bx}\right)dxdt+\rho_b\int_0^T\int_{\mathbb{R}^+} \left(b_2u^2+v^2\right)e^{2bx}dxdt\\
\vspace{3mm}\displaystyle -\left(\dfrac{3b}{2}-2\delta\right)\int_0^T\int_{\mathbb{R}^+} b_2u_x^2e^{2bx}dxdt-\left(3b\left(1-2a_3^2b_2\right)-2\delta\right)\int_0^T\int_{\mathbb{R}^+} v_x^2e^{2bx}dxdt.
\end{array}
\end{equation}
On the other hand, Theorem \ref{main-dec} guarantees the existence of a constant $C>0$ satisfying
\begin{equation}\label{estab8}
\begin{array}{l}
\vspace{3mm}\displaystyle \rho_b\int_{0}^T \int_0^{x_0} \left(b_2u^2+v^2\right)e^{2bx}dxdt\leq e^{2bx_0}\rho_b\int_{0}^T \int_0^{x_0} \left(b_2u^2+v^2\right)dxdt\\
\vspace{3mm}\displaystyle \leq \rho_b C\left\{\frac{1}{2}\int_{0}^T \left(\sqrt{b_2}u_x(0,t)+\sqrt{a_3^2b_2}v_x(0,t)\right)^2dt+\frac{1}{2}\left(1-a_3^2b_2\right)\int_{0}^T v_x^2(0,t)dt\right.\\
\vspace{3mm}\displaystyle\left.+\int_0^T\int_{\mathbb{R}^+} a(x)\left(b_2u^2+v^2\right)dxdt\right\}.
\end{array}
\end{equation}
Moreover, since $-a(x) \geq -a_0$ for $x \geq x_0$, it follows that
\begin{equation}\label{estab9}
\begin{array}{l}
\vspace{3mm}\displaystyle -\int_0^T\int_{\mathbb{R}^+} a(x)\left(b_2u^2+v^2\right)e^{2bx}dxdt\\
\vspace{3mm}\displaystyle= -\int_{0}^T \int_0^{x_0} a(x)\left(b_2u^2+v^2\right)e^{2bx}dxdt-\int_{0}^T \int_{x_0}^{+\infty}  a(x)\left(b_2u^2+v^2\right)e^{2bx}dxdt\\
\vspace{3mm}\displaystyle\leq -\int_{0}^T \int_0^{x_0} a(x)\left(b_2u^2+v^2\right)e^{2bx}dxdt-a_0\int_{0}^T \int_{x_0}^{+\infty}  \left(b_2u^2+v^2\right)e^{2bx}dxdt.
\end{array}
\end{equation}
Using \eqref{estab8} and \eqref{estab9} in \eqref{estab6} we obtain the estimate
\begin{equation}
\begin{array}{l}
\vspace{3mm}\displaystyle \mathcal{L}(U)-\mathcal{L}(U_0)\\
\vspace{3mm}\displaystyle \leq -\left(c_b+\dfrac{1}{2}-\dfrac{\rho_b}{2}C\right)\left(\int_{0}^T \left(\sqrt{b_2}u_x(0,t)+\sqrt{a_3^2b_2}v_x(0,t)\right)^2dt+\left(1-a_3^2b_2\right)\int_{0}^T v_x^2(0,t)dt\right)\\
\vspace{3mm}\displaystyle -\left(c_b-\rho_b C\right)\int_0^T\int_{\mathbb{R}^+} a(x)\left(b_2u^2+v^2\right)dxdt-\left(a_0-\rho_b\right)\int_{0}^T \int_{x_0}^{+\infty} \left(b_2u^2+v^2\right)e^{2bx}dxdt\\
\vspace{3mm}\displaystyle-\int_{0}^T \int_0^{x_0} a(x)\left(b_2u^2+v^2\right)e^{2bx}dxdt-\left(\dfrac{3b}{2}-2\delta\right)\int_0^T\int_{\mathbb{R}^+} b_2u_x^2e^{2bx}dxdt\\
\displaystyle-\left(3b\left(1-2a_3^2b_2\right)-2\delta\right)\int_0^T\int_{\mathbb{R}^+} v_x^2e^{2bx}dxdt,\nonumber
\end{array}
\end{equation}
from which we conclude that
\begin{equation}\label{estab10}
\begin{array}{l}
\vspace{3mm}\displaystyle \mathcal{L}(U)-\mathcal{L}(U_0)\\
\displaystyle\leq -k\left\{\int_{0}^T \left(\sqrt{b_2}u_x(0,t)+\sqrt{a_3^2b_2}v_x(0,t)\right)^2dt+\left(1-a_3^2b_2\right)\int_{0}^T v_x^2(0,t)dt+\right.\\
\displaystyle \left.\int_0^T\int_{\mathbb{R}^+} a(x)\left(b_2u^2+v^2\right)dxdt+\int_0^T\int_{\mathbb{R}^+}\left(b_2u_x^2+v_x^2\right)e^{2bx}dxdt\right\},
\end{array}
\end{equation}
where $k$ is a positive constant. Indeed, first we fix $\delta<\min\left\{\dfrac{3b}{4},\dfrac{3b\left(1-2a_3^2b_2\right)}{2}\right\}$ as in the proof of \eqref{12' b}. Then, we consider $c_b$ and $a_0$ satisfying
$$ a_0>\rho_b \mbox{ and } c_b>\rho_bC.$$
Finally, the proof of \eqref{semig-fin} is a consequence of \eqref{estab10} and the following results:
\begin{lemma}\label{claim7}
There exists a constant $C>0$, such that
\begin{equation}\label{c7}
\int_{0}^T \mathcal{L}(U)(t)dt\leq C \int_0^T\int_{\mathbb{R}^+} \left(b_2u^2_x+v^2_x\right)e^{2bx}dxdt.
\end{equation}
and
\begin{equation}\label{c8}
\begin{array}{l}
\vspace{3mm}\displaystyle \mathcal{L}(U_0)\leq C\left\{\int_{0}^T\left(\sqrt{b_2}u_x(0,t)+\sqrt{a_3^2b_2}v_x(0,t)\right)^2dt \right.\\
\vspace{3mm}\displaystyle \left.+\left(1-a_3^2b_2\right)\int_{0}^T v_x^2(0,t)dt+\int_0^T\int_{\mathbb{R}^+}
\left(b_2u_x^2+v_x^2\right)e^{2bx}dxdt+\int_{0}^T \mathcal{L}(U)(t)dt\right\}.
\end{array}
\end{equation}
\end{lemma}
\begin{proof} Inequality \eqref{c7} is a consequence of \eqref{estab -2}:
\begin{equation}\label{c7'}
\begin{array}{l}
\vspace{3mm}\displaystyle \int_{0}^T \mathcal{L}(U)(t)dt\leq \left(c_b+\dfrac{1}{2}\right)\int_0^T\int_{\mathbb{R}^+} \left(b_2u^2+b_1v^2\right)e^{2bx}dxdt\\
\vspace{3mm}\displaystyle\leq \dfrac{1}{b^2}\left(c_b+\dfrac{1}{2}\right)\left(1+b_1\right)\int_0^T\int_{\mathbb{R}^+} \left(b_2u_x^2+v_x^2\right)e^{2bx}dxdt.
\end{array}
\end{equation}
To obtain \eqref{c8} we proceed exactly as in the proof of \eqref{11'}. We multiply the first equation in \eqref{e1} by $\left(T-t\right)ue^{2bx}$, the second one by $\left(T-t\right)ve^{2bx}$ and add the resulting identities. After some integrations by parts  over $\RR^+\times\left(0,T\right)$, there appears the relation
\begin{equation}\label{c8 1}
\begin{array}{l}
\vspace{3mm}\displaystyle \frac{T}{2} \int_{\mathbb{R}^+} \left(b_2u_0^2(x)+b_1v_0^2(x)\right)e^{2bx}dx\\
\vspace{3mm}\displaystyle=\dfrac{1}{2}\int_0^T\int_{\mathbb{R}^+} \left(b_2u^2+b_1v^2\right)e^{2bx}dxdt-3b\int_0^T\int_{\mathbb{R}^+} \left(b_2u_x^2+v_x^2\right)\left(T-t\right)e^{2bx}dxdt\\
\vspace{3mm}\displaystyle+\frac{2b}{3}\int_0^T\int_{\mathbb{R}^+} \left(b_2u^3+v^3\right)\left(T-t\right)e^{2bx}dxdt\\
\vspace{3mm}\displaystyle-\dfrac{1}{2}\int_{0}^T \left(\sqrt{b_2}u_x(0,t)+\sqrt{a_3^2b_2}v_x(0,t)\right)^2\left(T-t\right)dt+\frac{1}{2}(1-a_3^2b_2)\int_{0}^T v_x^2(0,t)\left(T-t\right)dt\\
\vspace{3mm}\displaystyle-6ba_3b_2\int_0^T\int_{\mathbb{R}^+} u_xv_x\left(T-t\right) e^{2bx}dxdt+8b^3a_3b_2\int_0^T\int_{\mathbb{R}^+} uv\left(T-t\right)e^{2bx}dxdt\\
\vspace{3mm}\displaystyle-2ba_1b_2\int_0^T\int_{\mathbb{R}^+} \left(uv^2+vu^2\right)\left(T-t\right)e^{2bx}dxdt\\
\vspace{3mm}\displaystyle+br\int_0^T\int_{\mathbb{R}^+} v^2\left(T-t\right)e^{2bx}dxdt-\int_0^T\int_{\mathbb{R}^+} a(x)\left(b_2u^2+v^2\right)\left(T-t\right)e^{2bx}dxdt\\
\vspace{3mm}\displaystyle +\int_0^T\int_{\mathbb{R}^+} a(x)\left(b_2u^2+v^2\right)\left(T-t\right)dxdt.
\end{array}
\end{equation}
Now, we multiply the identity above by $T^{-1}$ and observe that $\frac{T-t}{T} < 1$. Then, using \eqref{a} and \eqref{cubico final} we can estimate straightforwardly in \eqref{c8 1} to obtain the inequality
\begin{equation}\label{c8 3}
\begin{array}{l}
\vspace{3mm}\displaystyle \dfrac{1}{2}\int_{\mathbb{R}^+} \left(b_2u_0^2(x)+b_1v_0^2(x)\right)e^{2bx}dx\\
\vspace{3mm}\displaystyle\leq \left(\dfrac{1}{2T}+4b^3\left|a_3\right|^2b_2\left(1+\dfrac{1}{b_1}\right)+br+2\left\|a\right\|_{L^\infty}\left(1+\dfrac{1}{b_1}\right)\right)\int_0^T\int_{\mathbb{R}^+} \left(b_2u^2+b_1v^2\right)e^{2bx}dxdt\\
\vspace{3mm}\displaystyle+\left(3b+2\delta+3b\left|a_3\right|^2b_2\right)\int_0^T\int_{\mathbb{R}^+} \left(b_2u_x^2+v_x^2\right)e^{2bx}dxdt\\
\displaystyle +\dfrac{1}{2}\int_{0}^T \left(\sqrt{b_2}u_x(0,t)+\sqrt{a_3^2b_2}v_x(0,t)\right)^2dt+\frac{1}{2}(1-a_3^2b_2)\int_{0}^T v_x^2(0,t)dt.
\end{array}
\end{equation}
From the definition of the function $\mathcal{L}=\mathcal{L}(t)$, we can conclude that
\begin{equation}\label{c8 4}
\begin{array}{l}
\vspace{3mm}\displaystyle \dfrac{1}{2}\int_{\mathbb{R}^+} \left(b_2u_0^2(x)+b_1v_0^2(x)\right)e^{2bx}dx\\
\vspace{3mm}\displaystyle \leq C'\left\{\int_{0}^T\left(\sqrt{b_2}u_x(0,t)+\sqrt{a_3^2b_2}v_x(0,t)\right)^2dt+\left(1-a_3^2b_2\right)\int_{0}^T v_x^2(0,t)dt \right.\\
\vspace{3mm}\displaystyle \left.+\int_0^T\int_{\mathbb{R}^+} \left(b_2u_x^2+v_x^2\right)e^{2bx}dxdt+\int_{0}^T \mathcal{L}(U)(t)dt\right\},
\end{array}
\end{equation}
where $C'$ is a positive constant. Since
$$\mathcal{L}(U_0)\leq \left(\dfrac{1}{2}+c_b\right) \int_{\mathbb{R}^+} \left(b_2u_0^2(x)+b_1v_0^2(x)\right)e^{2bx}dx,$$
the result follows.
\end{proof}
Now the proof of Theorem \eqref{teo-dec} is complete.
\end{proof}

\

\noindent{\bf Acknowledgments:} AFP author was partially
supported by CNPq (Brazil). GRS was partially supported by Capes (Brazil).

\label{referencias}


\begin{thebibliography}{99}
\bibitem{ablowitz} M. Ablowitz, D. Kaup, A. Newell and H. Segur, \, {\it Nonlinear-evolution equations of
physical signifcance}, \, Phys. Rev. Lett. {\bf 31} (1973), 125--127.

\bibitem{alarcon} E. Alarcon, J. Angulo and J. F. Montenegro, \, {\it Stability and instability of solitary waves for a
nonlinear dispersive system}, \, Nonlinear Anal. {\bf 36} (1999), 1015--1035.

\bibitem{vep} E. Bisognin, V. Bisognin and G.P. Menzala, \, {\it Exponential stabilization of a
coupled system of Korteweg-de Vries Equations with localized
damping}, \, Adv. Diff. Eq., {\bf 8} (2003), 443--469.

\bibitem{bona-winther} J. Bona, Jerry and R. Winther, \, {\it The Korteweg-de Vries equation, posed in a quarter-plane}, \, SIAM J. Math. Anal. {\bf 14} (1983), 1056-�1106.

\bibitem{bona} J. Bona, G. Ponce, J. C. Saut and M. M. Tom, \,
{\it A model system for strong interaction between internal solitary
waves}, \, Comm. Math. Phys., {\bf 143} (1992), 287--313.



\bibitem{cavalcanti} M. M. Cavalcanti, V. N. D. Cavalcanti, A. Faminskii and F. Natali,
\, {\it Decay of solutions to damped Korteweg�de Vries type equation}, Appl. Math. Optim. {\bf 65} (2012), 221--251.

\bibitem{cerpa-pazoto} E. Cerpa and A. F. Pazoto, \, A note on the paper ``{\it On the controllability of a coupled system of two Korteweg-de Vries equations}'' [MR2561938], \, Commun. Contemp. Math. {\it 13} (2011), 183-�189.



\bibitem{davila} M. Davila, {\it On the unique continuation property for a coupled system of
Korteweg-de Vries equations}, PhD Thesis, Institute of Mathematics,
Federal University of Rio de Janeiro, Brazil, (1994).



\bibitem{gg} J. A. Gear and R. Grimshaw, \, {\it Weak and strong interaction
between internal solitary waves}, \, Stud. in Appl. Math., {\bf 70}
(1984), 235--258.


\bibitem{kato} T. Kato,\, {it  On the Cauchy problem for the (Generalized) Korteweg-de Vries equation}, \, Stud. Appl. Math. Adv. Math. Suppl. Stud., {\bf 8} (1983), 93�128.


\bibitem{laurent} C. Laurent, L. Rosier and B.-Y. Zhang,\, {\it Control and stabilization of the Korteweg-de Vries equation on a periodic domain},\, Comm. Partial Differential Equations {\bf 35} (2010), 707-�744.

\bibitem{li} F. Linares and M. Panthee, \, {\it On the Cauchy problem for a coupled
system of KdV equations}, \, Commun. Pure Appl. Anal., {\bf 3}
(2004), 417--431.

\bibitem{linares-pazoto} F. Linares and A. F. Pazoto, \, {\it Asymptotic behavior of the Korteweg-de Vries equation posed in a quarter plane},\, J. Differential Equations {\bf 246} (2009), 1342�1353.

\bibitem{linares-pazoto1} F. Linares and A. F. Pazoto, \, {\it On the exponential decay of the critical
generalized Korteweg-de Vries equation with localized damping}, \, Proc. Amer. Math. Soc. {\bf 135}
(2007), 1515--1522.

\bibitem{mp} C. P. Massarolo and A. F. Pazoto, \, {\it Uniform stabilization of a nonlinear
coupled system of Korteweg-de Vries equation as a singular limit of the Kuramoto-Sivashinsky system}, \,
Differential Integral Equations {\bf 22} (2009), 53--68.

\bibitem{massarolo-menzala-pazoto} C.P. Massarolo, G. P. Menzala, and A.F. Pazoto, \, {\it Uniform stabilization of a class of KdV equations with localized damping}, \, Quarterly of Appl. Math., {\bf 69} (2011), 723--746.

\bibitem{massarolo-menzala-pazoto1} C. P. Massarolo, G. P. Menzala, A. F. Pazoto, \, {\it On the uniform decay for the Korteweg-de Vries equation with weak damping}, \, Math. Methods Appl. Sci. {\bf 30} (2007), 1419�-1435.

\bibitem{mvz} G. P. Menzala, C.F. Vasconcellos and E. Zuazua, \, {\it Stabilization of the Korteweg-de Vries
equation with localized damping}, \, Quarterly of Appl. Math., {\bf LX} (1)
(2002), 111--129.

\bibitem{mo} S. Micu and J. H. Ortega, \, On the controllability of
a linear coupled system of Korteweg-de Vries equations, \, in {\it
Mathematical and numerical aspects of wave propagation} (Santiago
de Compostela, 2000), \, SIAM, Philadelphia, PA (2000), 1020--1024.

\bibitem{mop} S. Micu, J. H. Ortega and A. F. Pazoto, \, {\it On the controllability of
a nonlinear coupled system of Korteweg-de Vries equations}, \,
Commun. Contemp. Math., {\bf 11} (5) (2009), 799--827.


\bibitem{dugan} D. Nina, A. F. Pazoto and L. Rosier, \, {\it Global stabilization of a coupled system of two generalized Korteweg-de Vries type equations posed on a finite domain}, Math. Control Relat. Fields {\bf 1} (2011), 353--389.

\bibitem{pazoto-souza} A. F. Pazoto and G. R. Souza, \, {\it Uniform stabilization of a nonlinear dispersive system},\,Quarterly of Appl. Math., To appear.


\bibitem{pazoto} A. F. Pazoto, \, {\it Unique continuation and decay for the
Korteweg-de Vries equation with localized damping}, \, ESAIM Control Optim. Calc. Var. {\bf 11} (2005),
473--486.

\bibitem{pazoto-rosier} A. F. Pazoto and L. Rosier, \, {\it Uniform stabilization in weighted Sobolev spaces for the KdV equation posed on the half-line},\, Discrete Contin. Dyn. Syst. Ser. B 14 (2010), 1511--1535.

\bibitem{rosier1} L. Rosier, \, {\it Exact boundary controllability for the Korteweg-de Vries
equation on a bonded domain}, \, ESAIM Control Optimization and
Calculus of Variations, 2 (1997), 33--55.

\bibitem{rosier2} L. Rosier, \, {\it Control of the surface of a fluid by a
wavemaker}, \, ESAIM Control Optimization and Calculus of Variations
{\bf 10} (2004), 346--380.

\bibitem{rosier3} L. Rosier, {\it Exact boundary controllability for the linear
Korteweg-de Vries equation on the half-line}, SIAM J. Control Optim.
{\bf 39}  (2000),  no. 2, 331--351.

\bibitem{rosier-zhang} L. Rosier and B.-Y. Zhang, \, {\it Global stabilization of the generalized Korteweg-de Vries
equation posed on a finite domain},\, SIAM J. Control Optim. {\bf 45} (2006), 927-956.

\bibitem{saut} J.-C. Saut and N. Tzvetkov, \, {\it On a model system for the oblique
interaction of internal gravity waves}, \, M2AN Math. Model. Numer. Anal. {\bf 34} (2000), 501�-523.




\bibitem{vera} O. P. Vera Villagran, {\it Gain of regularity of the solutions of a coupled
system of equations of Korteweg-de Vries type}, PhD Thesis,
Institute of Mathematics, Federal University of Rio de Janeiro,
Brazil, (2001).



\end{thebibliography}
\end{document}